\documentclass[12pt]{amsart}
\usepackage[utf8]{inputenc}
\usepackage[english]{babel}
\usepackage[a4paper,twoside,top=1.2in, bottom=1.1in, left=0.8in, right=0.8in]{geometry}

\usepackage{graphicx}
\usepackage{caption} 

\usepackage{amsmath}
\usepackage{amsthm}
\usepackage{amssymb}
\usepackage{esint} 
\usepackage{mathrsfs} 
\usepackage{faktor} 
\usepackage{dsfont} 
\usepackage{xcolor}
\usepackage{array}
\usepackage{hhline}
\usepackage{enumerate}
\usepackage{comment} 
\usepackage{imakeidx} 
\usepackage{xparse} 
\usepackage{mathtools}
\usepackage{fancyhdr} 
\usepackage{ifthen} 
\usepackage{forloop} 
\usepackage{xstring}
\usepackage{emptypage} 
\usepackage{minibox}

\usepackage[initials]{amsrefs} 

\usepackage{tikz}
\usetikzlibrary{arrows,shapes,patterns,calc,fadings,decorations.pathreplacing,cd}

\usepackage{hyperref} 
\hypersetup{
	colorlinks = true,
	linkcolor = {blue},
	urlcolor = {red},
	citecolor = {blue},
}

\newcounter{results}[section] 

\newcounter{steps}[section] 

\theoremstyle{plain}
\newtheorem{theorem}[results]{Theorem}

\newtheorem{lemma}[results]{Lemma}
\newtheorem{proposition}[results]{Proposition}
\newtheorem{corollary}[results]{Corollary}

\newtheorem*{theorem*}{Theorem}
\newtheorem*{lemma*}{Lemma}
\newtheorem*{proposition*}{Proposition}
\newtheorem*{corollary*}{Corollary}
\newtheorem*{exercise*}{Exercise}
\newtheorem*{fact*}{Fact}

\theoremstyle{remark}
\newtheorem{remark}[results]{Remark}

\theoremstyle{definition}
\newtheorem{definition}[results]{Definition}

\newtheorem*{definition*}{Definition}
\newtheorem*{example*}{Example}

\numberwithin{equation}{section}

\newcommand{\R}{\ensuremath{\mathbb R}}

\DeclareMathOperator{\tr}{tr}

\newcommand{\lapl}{\ensuremath{\Delta}} 

\DeclareMathOperator{\II}{I\!I} 

\newcommand{\M}{\ensuremath{\mathscr M}} 

\newcommand{\met}{\ensuremath{\mathscr R}} 

\newcommand{\e}{\varepsilon}
\newcommand{\be}{\begin{equation}}
\newcommand{\ee}{\end{equation}}

     \title[Attaching faces of PSC manifolds with corners]{Attaching faces of positive scalar curvature manifolds with corners}
     \author{Alessandro Carlotto and Chao Li}
     \address{
     \newline \indent Alessandro Carlotto: 
     \newline Universit\`a di Trento, Dipartimento di Matematica,
via Sommarive 14, 38123 Trento, Italy
\newline\textit{E-mail address: alessandro.carlotto@unitn.it}
     	\newline ETH D-Math, R\"amistrasse 101, 8092 Z\"urich, Switzerland 
     	 \newline \textit{E-mail address: alessandro.carlotto@math.ethz.ch} 
     	 \newline \newline \indent Chao Li: 
     	 \newline New York University - Courant Institute of Mathematical Sciences, 251 Mercer Street, New York, NY 10012, United States of America
     	 \newline \textit{E-mail address: chaoli@nyu.edu} 
       }
     
\begin{document}

     \dedicatory{In memory of Robert Bartnik}
     	
     	\begin{abstract}
     	We prove a novel desingularization theorem, that allows to smoothly attach two given manifolds with corners by suitably gluing a pair of isometric faces, with control on both the scalar curvature of the resulting space and the mean curvature of its boundary.
	\end{abstract}

     	\maketitle      
     
     \tableofcontents

     \section{Introduction} \label{sec:intro}

In 1989 Bartnik proposed a novel, intriguing notion of quasi-local mass \cite{Bar89}, which would propel a tremendous amount of mathematical research in the following decades. Roughly speaking, and sticking to the simplest possible setting, he suggested to quantify the amount of \emph{gravitational mass} in a given region $\Omega$ of a time-symmetric, asymptotically flat initial data set $(M,g)$ by considering all \emph{admissible} asymptotically flat extensions of such a region and then minimizing the ADM mass functional within such a class. Here, the word ``admissible'' encodes on the one hand the curvature requirement(s) imposed by the Einstein constraint equations (cf. e.\,g. \cite{Car21b}), and on the other hand additional geometric conditions, such as for instance the prescription that $\partial\Omega$ not be enclosed by apparent horizons (or other, similar yet in general inequivalent non-degeneracy conditions aimed at ensuring that the infimum in question not be trivially zero, see \cite{McC20}). We refer the reader to \cite{Bar02}, \cite{Shi12} and \cite{WanYau09, WanYau09b}, among others, for a comparative discussion of various notions of quasi-local mass within the framework of general relativity, with main focus on asymptotically flat initial data sets.

While enjoying a number of physically desirable properties, Bartnik's definition immediately poses some remarkable challenges on the analytic front, which naturally stem from its very variational character. Firstly, one would like to understand to what extent the degree of \emph{regularity} of the aforementioned extensions plays a role in minimizing the functional in question. Secondly, is it plausible to expect existence of minimizers (that is to say: to prove that the infimum above is actually achieved) for ample classes of data?
It is clear that such two questions are very much intertwined, for indeed - say by analogy with all classic problems in the calculus of variations - 
it would be natural to pose the problem above under a lower regularity requirement and then prove \emph{a posteriori} the regularity of extremals (if feasible).
Both questions are arguably challenging when posed in full generality, but the partial advances we have witnessed have nevertheless shed new light on this matter, sometimes leading to far-reaching and partly unexpected developments. (For a very recent review of the state of the art concerning the Bartnik quasi-local mass conjectures the reader may wish to consult \cite{And23}.)

Concerning the first issue, the most basic singularities one may allow for are those occurring at the interface of the domain under consideration. Motivated by the problem of proving the Riemannian positive mass theorem for metrics that may only have Lipschitz regularity along a closed hypersurface (and are smooth anywhere else), Miao designed in \cite{Mia02} a method for desingularizing such interfaces in a controlled fashion while keeping the scalar curvature non-negative, as prescribed by the dominant energy condition in the time-symmetric case; interestingly, it is not necessary for the mean curvature functions on both sides of the interface to match, but it is simply enough that they satisfy the pointwise inequality $H_{-}\geq H_{+}$ (condition $(\textbf{H})$ therein), which one may regard as the requirement of a positive contribution to the scalar curvature, in a suitable weak (in fact: distributional) sense.  This type of smoothing procedure has, since then, become a standard tool in geometric analysis: among various applications (and variations on the theme), it has been crucially employed by the authors to study, on any compact 3-manifold, the space of Riemannian metrics having positive scalar curvature and mean-convex boundary \cite{CarLi19, CarLi21}, proving that such a space is always either empty or contractible. In general terms, the possiblity of \emph{doubling} manifolds with boundary (while preserving suitable curvature conditions) allows to rephrase the original problem into one for closed manifolds, albeit in presence of a $\mathbb{Z}_2$-equivariance constraint given by a reflectional symmetry.

Our purpose here is to study to what extent this result by Miao  can be extended to the case of \emph{manifolds with boundary}; this amounts to investigating the problem of attaching manifolds with corners (see \cite{Joy12} for an interesting overview, and for various background references) by smoothly gluing a pair of distinguished faces while preserving some pre-assigned curvature conditions; typically the conditions we focus on are the positivity of the scalar curvature of the resulting manifold and (for instance) the minimality or mean-convexity of its (leftover) boundary. In particular, we shall be concerned with the case when the isometric faces to be attached form, in their respectively manifolds, a right angle with the adjacent faces; an important special case, partly motivated by forthcoming geometric applications in the spirit of \cite{BamLiMan22}, is when the interfaces in question are free boundary minimal hypersurfaces (see also our recent note \cite{CarLi23}, where we introduced the notion of \emph{minimal concordance}, for related discussions). The main theorem we prove, that is Theorem \ref{thm:Desing}, is rather general and structured, so that the reader may wish to first look at Corollary \ref{cor:HalfSphere} for a simpler, illustrative application. 

After the original proof by Miao, which proceeds by fiberwise convolution and then conformal deformation, the same result has also been obtained via different methods: specifically via Hamilton's Ricci flow in \cite{McFSze12} (also with the  aim of proving the positive mass theorems for metrics with edge singularities along a hypersurface) and then, much more recently, with refined local deformation methods in \cite{BarHan20}. Our approach is arguably closest to the third and latest of the methods above, although it also exploits (in the very final step) a localized conformal deformation and thus is somewhat hybrid in spirit; we also note some analogies with the important work by Brendle-Marques-Neves \cite{BreMarNev11}, where counterexamples to the Min-Oo conjectures have first been constructed. A specific challenge we face is that, since the interface in question has a boundary (hence, ultimately, since the distance function from such an interface may not be smooth and, no matter that, its level sets will not in general meet the ambient boundary orthogonally) we need to handle Riemannian metrics that take the rather general form $g= u(x,t)^2dt^2 + h_t(x)$ near $X$, to be contrasted to the important special case when $u\equiv 1$ that is indeed the object of \cite{BarHan20}.

Our main result is proven in Section \ref{sec:Desing}, although we decided to 
single out two preparatory reductions deserving separate statements (and correspondingly devoted sections)  since we believe that some of such arguments may be of independent interest and utility, and then to split the proof in question into two separate steps; in particular, the problem of connecting two metrics whose warping factors do not match (cf. Remark \ref{rem:WarpFactNoMatch}), which is perhaps the most challenging part of the whole work, is studied in Step 1 therein (crucially building on the \emph{ad hoc} fairly delicate construction of a suitable interpolation function presented in Appendix \ref{app:GluingFunction}).

Lastly, in Section \ref{sec:Applications} we collected some direct yet significant applications, also partly connected to and motivated by our recent note \cite{CarLi23}. In particular, we justify for instance the following claims: firstly, the relation of weak PSC min-concordance (as per Definition 2.3 therein) is \emph{transitive} hence an equivalence relation; secondly,
if one restricts a priori to the subspace $\met_{R>0, H=0}(X)$ then two metrics are PSC min-concordant if and only if they are \emph{weakly} PSC min-concordant. Said otherwise,
it is the same to require, within that subspace and on top of conditions (i) and (ii) of that definition, for item (iii) therein that the faces $X\times\left\{0\right\}$ and $X\times\left\{1\right\}$ meet $\partial X\times [0,1]$ orthogonally and satisfy any of the (local) geometric conditions in the following hierarchy: 
$
\centering
(\text{doubling}) \ \Rightarrow \ (\text{totally geodesic}) \ \Rightarrow \  (\text{minimal}),$
or even that the metric be a Riemannian product in a neighborhood of the faces.

The relevance of the notion of weak concordance has been justified in \cite{CarLi23} (which in turn stems from earlier work of Mantoulidis-Schoen \cite{ManSch15} on the Bartnik mass of apparent horizons), in view of its relation with 
 spaces of metrics defined by ``spectral stability'' conditions.
 Such spaces $\M^{>0}_{\kappa}(X)$ (and in particular the central $\M^{>0}_{1/2}(X)=:\M(X)$) are in fact much larger than $\met_{R>0,H=0}(X)$ and, correspondingly, enjoy additional degrees of flexibility that are desirable in various situations. Thanks to Corollary \ref{cor:Equivalence}, Definition 2.3 in \cite{CarLi23} thus emerges as an arguably reasonable extension of the classical, fundamental notion of concordance to such spaces.

\section{The ambient metric near a free boundary minimal hypersurface }\label{sec:NormalForm}

Given $n\geq 2$, we let $X^n$ denote a compact differentiable $(C^{\infty})$ manifold, of dimension equal to $n$, with possibly non-empty boundary. Consistently with our previous work \cite{CarLi21}, we will denote by $\met=\met(X)$ the cone of smooth Riemannian metrics on $X$, and we shall be particularly concerned with its topological subspaces defined by binary relations involving its scalar curvature and the mean curvature of its boundary. 

For a Riemannian metric $h$ on $X$ we let $\nu=\nu_h$ denote an outward-pointing unit normal vector field along $\partial X$, take $\II_h$ to be the second fundamental form (with respect to $\nu$) and $H_h$ its trace (that is: the mean curvature of $\partial X$); throughout this article, we adopt the convention that the unit sphere in $\R^3$ has mean curvature equal to $2$
and we will say - for a Riemannian metric $h$ on $X$ - that $(X,h)$ is mean-convex if the mean curvature of such a manifold is greater or equal than zero, namely if $H_{h}\geq 0$; when we wish to stress the strict inequality $H_h>0$ we will write it explicitly. Geometrically speaking, with our convention strict mean-convexity (of a boundary) means that an outward deformation will increase area to first order.

For later reference and use, we recall some facts from \cite{CarLi23}. Firstly, we collect in the next statement some basic slicing formulae concerning warped product metrics on cylinders. (Here, and in the sequel, we will informally employ the word \emph{cylinder} when referring to any smooth manifold of the form $X\times J$ for any interval $J\subset\R$; the cylindrical boundary is then by definition $\partial X\times J$.)

\begin{lemma}\label{lem:Formulae}
Let us consider on the product manifold $M=X\times J$ a smooth metric of the form
\[
g(x,t)=u(x,t)^2 dt^2+h_t(x),
\]
where $u\in C^{\infty}(M)$ and the map $J\ni t \mapsto h_t(x)\in  \met(X)$ is also smooth.
 Then the following formulae hold:
\begin{enumerate}[(1)]
\item{2nd fundamental form of the slice $X\times\left\{t\right\}$
\[\II_t(x)=(2u(x,t))^{-1}\frac{d}{dt}h_t(x);\]}
\item{mean curvature of the slice $X\times\left\{t\right\}$
\[
H_t(x)=(2u(x,t))^{-1}\tr_{h_t}\frac{d}{dt}h_t(x);\]}
\item{scalar curvature of the product manifold
\[
R_g(x,t) =2u(x,t)^{-1}(-\lapl_{h_t} u+\frac{1}{2}R_{h_t} u)
 -2u(x,t)^{-1}\frac{d}{dt}H_t(x)-(H_t(x))^2-|\II_t|^2
\]
}
\item{mean curvature of the cylindrical boundary of the product manifold
\[
H_g(x,t)=H_h (x,t)  +\partial_{\nu_h}\log u(x,t).
\]}
\end{enumerate} 
(Note that, for the first two equations we have considered $X\times\left\{t\right\}$ as boundary of $X\times [0,t]$, i.\ e. we worked with respect to the normal $\partial_t$).
\end{lemma}

\begin{lemma}\label{lem:Harmonic}
  Let $(M^{n+1},g)$ be a Riemannian manifold and let $X$ be a compact, connected, properly embedded, two-sided  hypersurface meeting the ambient boundary orthogonally. Then there exist a smooth ($C^{\infty}$) function $f: M\to\mathbb{R}$ such that the following conditions hold true:
  \begin{enumerate}[(1)]
      \item $f=0$ on $X$;
      \item the restriction of the $g$-gradient $[\nabla f]_{|\partial M}$ is tangent to $\partial M$ at each point;
      \item the $g$-gradient $\nabla f$ does not vanish at any point of $X$.
  \end{enumerate}
\end{lemma}

A construction of this sort is presented, for instance, in \cite[Lemma 2.3]{Aim23}; in fact, we are not even interested in property (iii) therein. While the setting the author refers to is Euclidean $\R^N$, it is easy to check that the result can be transplanted, with rather obvious modifications, to general Riemannian manifolds. We will exploit the following consequence.

\begin{proposition}\label{prop:NormalForm}
 Let $(M^{n+1},g)$ be a Riemannian manifold and let $X$ be a compact, connected, properly embedded, two-sided hypersurface meeting the ambient boundary orthogonally. Then there exists an open neighborhood $U\supset X$ and a diffeomorphism $\Phi: X\times (-\delta,\delta)\to U$ (for some $\delta>0$) 
    such that the pulled-back metric $\Phi^* g_{|U}$ takes the form $u(x,t)^2 dt^2+h_t(x)$ for smooth $u\in C^{\infty}(X\times (-\delta,\delta))$ and $h\in C^{\infty}((-\delta,\delta), \met(X))$.
\end{proposition}

\begin{proof}
   Let $f:M \to\R$ be as in the preceding lemma, and let then $\Phi$ denote the gradient flow of $f$, namely the flow whose velocity at each point equals the $g$-gradient of $f$. Note that by property 
   (2) above such a flow is well-defined (for all times) in spite of the boundary of $M$, and by
   (3) there exists an open neighborhood $U$ of $X$ where $\nabla f$ does not vanish at any point, and thus the level sets of $f$ give a foliation of $U$.
   
   Thanks to Lemma \ref{lem:Harmonic} $\Phi$ determines a smooth diffeomorphism $X\times (-\delta,\delta)\to U$, and so $\Phi^* g_{|U}$ will have smooth coefficients.  We observe that, by construction, for any index $i\in \left\{1,\ldots, n\right\}$ the vector field $\Phi_{\ast}\partial _{i}$ is, say at a point $(x,t)$, tangent to the corresponding level set of $f$, while the vector field $\Phi_{\ast}\partial_t$ coincides with the gradient of the same function: hence, 
$\Phi^* g (\partial_{i},\partial_t) =0$ at each point $(x,t)\in X\times [0,\delta)$. Thereby, the conclusion follows by simply letting $u(x,t)^2=\Phi^* g (\partial_{t},\partial_t)$ and $h_t(x)=\Phi^* g (\partial_{i},\partial_j)$ as one varies $i,j=1,2,\ldots, n$.
\end{proof}

We note, as an important special case, that the proposition above applies to two-sided free boundary minimal hypersurfaces.

\begin{remark}\label{rem:NecessityFreeBoundary}
We explicitly remark that the ambient metric $g$ having a block form as stated in Proposition \ref{prop:NormalForm} forces, in particular, $X$ to meet $\partial M$ at a right angle according to the metric $g$. In this sense, the content of such a statement is \emph{de facto} a characterization.
\end{remark}

\begin{remark}\label{rem:ObstructionsNormalForm}
One may in fact wish to impose additional requirements on the warping factor $u$ mentioned in Proposition \ref{prop:NormalForm} above, but there are definite obstructions to deal with. In particular, as the following two examples show, differently from the closed case one \emph{cannot} in general expect - no matter what sort of refinement of the construction - in any such local form $u(x,0)$ to be constant in $x\in X$, nor the weaker ``compatibility condition'' $\partial_{\nu}u(x,0)=0$ to hold for $x\in \partial X$:
\begin{itemize}
\item{consider $M^3=\mathbb{B}^3$ the (closure of the) unit ball in Euclidean space $\R^3$, and let $X$ denote the standard unit disk: then (in the notation of Lemma \ref{lem:Formulae}) one has $H_g=2$, $H_h=1$ and so  by item (4) therein $\partial_{\nu}\log u(x,0)=1$ at any point of the unit circle $\partial X$;}
\item{consider $M^3$ the (closure of the) connected component in Euclidean space $\R^3$ that contains the origin among those two bounded by the standard catenoid of unit waist, and let $X$ denote again the standard unit disk: then one has $H_g=0$, $H_h=1$ and so $\partial_{\nu}\log u(x,0)=-1$ at any point of the unit circle $\partial X$.}
\end{itemize}
\end{remark}

\begin{remark}\label{rem:Scaling}
 We however note, for later reference, that if $f$ satisfies the conditions of Lemma \ref{lem:Harmonic} then so will $\lambda f$ for any positive constant $\lambda$; furthermore, the resulting metric factor $u=u(x,t)$ gets also rescaled by the same number. Hence, we can ensure that $u(x,0)>0$ is made arbitrarily large, or arbitrarily close to zero, if needed.
\end{remark}

\section{Deformation to $C$-normal form}\label{sec:Cnormal}

From now onward we shall freely adopt and employ the language of manifolds with corners, cf. \cite{Joy12} and references therein. Given $M^{n+1}$ a manifold with corners, of dimension $n+1\geq 3$, its singular locus (henceforth denoted $\text{sing}(M)$) is here understood as the set of points around which $M$ is not modelled by $\R^{n+1}$ or $\R^{n}\times\R^+$ (for $\R^+=[0,\infty)$); a face is understood as the closure of a connected component of $\partial M\setminus \text{sing}(M)$. In that setting we will say that $F$ is a cylindrical face (and that $M$ is cylindrical about $F$) if there exists an open set $U\supset F$ that is mapped diffeomorphically (in the category of manifolds with corners) to a product $X\times \R^+$ for some smooth, compact manifold with boundary $X$. (Note that this happens if and only if $\text{sing}(M)\cap X$ is a smooth codimension-two submanifold of $M$.)

To proceed, note that Lemma \ref{lem:Harmonic} and Proposition \ref{prop:NormalForm} can conveniently be applied to this setting. So let $M$ be as above, let $F$ be a cylindrical face and assume we are given a Riemannian metric $g$ on $M$ satisfying the assumption 
\[
F \ \text{meets the adjacent face(s) at a right angle} \quad \quad (\boldsymbol{\perp}).
\]
Then we can write $g$, in a neighborhood of $F$ as $g(x,t)=u^2(x,t)dt^2+h_t(x)$ where say $0\leq t\leq t_0$, for smooth $u$ and $h$. In fact, all deformations we present in the present section, as well as the in the next one, are purely local and so - for the sake of simplicity - we will assume our ambient manifold is just $M=X\times \R^+$ (and, with abuse of notation, we let $F$ denote $X\times\left\{0\right\}$, so that the adjacent faces are identified with $\partial X\times \R^+$).

We will work with Banach spaces (of real-valued functions, or tensors) of the form $C^{k,\alpha}(X)$; the length of a tensor being measured with respect to a background metric that we shall specify at due course. Furthermore, for given $t_0>0$ and $\ell\geq 0$ we will also deal with the ``parabolic'' counterpart, namely with the spaces $C^{\ell}([0,t_0], C^{k,\alpha}(X))$ under the same caveat as above; the corresponding norm will be denoted by
\[
\|\cdot\|_{C^{\ell}(t;C^{k,\alpha}(X))}
\]
with the interval $[0,t_0]$ specified once and for all. Given a smooth map $J\ni t\mapsto h_t\in\met(X)$ the first (respectively: second) derivatives with respect to the variable $t$ will be denoted by $\dot h_t$ (respectively: $\ddot h_t$); similarly, for $u=u(x,t)$ we write $\dot u_t$ and $\ddot u_t$ for the first and second derivatives in $t$.

That being said, by Lemma \ref{lem:Formulae} the second fundamental form and the mean curvature of a slice $\{t=constant\}$, taken with respect to the unit normal $-u^{-1}\partial_t$, are given by
\[\II_t = -(2u)^{-1} \dot h_t, \quad H_t = -(2u)^{-1}\tr_{h_t}\dot h_t.\]
Thus, 
\begin{align*}
    \dot H_t &= -\frac{d}{dt} \left(\frac{1}{2u} \tr_{h_t} \dot h_t\right)\\
             &= \left(\frac{\dot{u_t}}{2u^2}\right)\tr_{h_t}\dot h_t - \frac{1}{2u} |\dot h_t|^2 - \frac{1}{2u} \tr_{h_t} \ddot h_t,
\end{align*}
and (note here that the term involving the derivative of mean curvature changes sign compared to Lemma \ref{lem:Formulae}, where one works with respect to $\partial_t$)
\begin{equation}\label{eq:warped.BarHanke.scalar.curvature}
    \begin{split}
        R_g &= R_{h_t} - \frac{2\Delta_{h_t}u}{u} + \frac{2\dot H_t}{u} - H_t^2 - |\II_t|^2 \\
        &= R_{h_t} - \frac{2\Delta_{h_t}u}{u} + \frac{\dot{u_t}}{u^3} \tr_{h_t}\dot h_t - \frac{5}{4u^2}|\dot h_t|^2 - \frac{1}{u^2} \tr_{h_t} \ddot h_t - \frac{1}{4u^2}\left(\tr_{h_t}\dot h_t\right)^2.    
    \end{split}
\end{equation}

 Following \cite{BarHan20}, we then recall:
\begin{definition}\label{def:Cnormal}
    In the setting above, we say that a metric $g$ is $C$-normal near $F$, if it takes the form
    \[g(x,t) = u(x,t)^2 dt^2 + h_0(x) - t h_1(x) - Ct^2 h_0(x),\]
    in an open neighborhood of $F$.
\end{definition}

Here is the key statement we prove in this section:

\begin{proposition}\label{prop:C_normal_deformation}
    Given a manifold with corners $M$, a designated cylindrical face $F$ and a Riemannian metric $g$ satisfying $(\boldsymbol{\perp})$, an open neighborhood $U$ of $F$, and $\eta>0$,  there exists $C_0=C_0(g)$ such that, for every $C\ge C_0$ there exists a $C$-normal metric $\hat g$ on $M$ such that:
    \begin{enumerate}[(1)]
        \item $\hat g = g$ in $M\setminus U$;
        \item $\hat g|_F = g|_F$;
        \item $\hat\II = \II$ on $F$;
        \item $\hat g - g$ has no $dt$ factor in the coordinates $(x,t)$;
        \item $\|\hat g - g\|_{C^1(M)} < \eta$;
        \item $R_{\hat g}>R_g - \eta$.
    \end{enumerate}
\end{proposition}

(In item (2) we really mean the restriction of both metrics, for each point of $F$, to the subspace of tangent vectors to $F$ itself; in item (4) we mean that in the local coordinates $(x,t)$ the coefficient of $dt^2$ or any $dtdx^i$ for the tensor $\hat g - g$ is identically equal to zero in $U$.)

We need the following lemma concerning the design of a suitable logarithmic cutoff function.

\begin{lemma}\label{lem:BarHanke_testfunction}
    For any $\delta\in (0,1/4)$, $\varepsilon\in (0,1)$, there exists a $C^\infty$ function $\tau_{\delta,\varepsilon}:\R\to \R$ such that:
    \begin{enumerate}[(1)]
        \item $\tau_{\delta,\varepsilon}=1$ when $t\le \delta\varepsilon$.
        \item $\tau_{\delta, \varepsilon}=0$ when $t\ge \varepsilon$.
        \item $0\le \tau_{\delta,\varepsilon}\le 1$ for all $t\in \R$.
        \item For every positive integer $\ell$, there is a constant $C_{\ell}>0$ such that for all $t>0$, 
        \[\left| \tau_{\delta,\varepsilon}^{(\ell)}(t)\right|\le C_{\ell} \cdot t^{-\ell}\cdot |\log\delta|^{-1}.\]
    \end{enumerate}
\end{lemma}
For a detailed proof, see \cite[Appendix B]{BarHan22local}.

\begin{proof}[Proof of Proposition \ref{prop:C_normal_deformation}]
    Take the Taylor expansion of $h_t$, seen as a $C^2$ map $t\mapsto h_t\in C^2(M,h_0)$: this reads
    \[h_t(x) = h_0(x) + t\dot h_0(x) + \frac12 t^2 \ddot h_0(x)+Q_t(x),\]
    where the remainder term $Q_t$ satisfies $\|Q_t\|_{C^2(M,h_0)} = o(t^2)$ as $t\to 0^+$.

    We then consider the following auxiliary metrics: for $s\in [0,1]$, define
    \[g^{(s)}(x,t) = g(x,t) - s\left(\frac12 t^2(\ddot h_0(x) + 2C h_0(x)) + Q_t(x)\right).\]
    $g^{(s)}$ is not the metric we need at the end of the construction, but serves as a good comparison for scalar curvature. For indeed, observe that for all $s\in [0,1]$, $g^{(s)}$ and $g$ have the same first order terms in $t$ and so, in particular, the scalar curvature formula above (equation \eqref{eq:warped.BarHanke.scalar.curvature}) implies that
    \[R_{g^{(s)}}(x,0) = R_{g}(x,0)+ \frac{s}{u^2}(\tr_{h_0} \ddot h_0+2n C)\ge R_{g}(x,0),\]
    where the last inequality holds provided that the constant $C$ is taken sufficiently large (depending on $h_0$ and $\ddot h_0$).

    Now let $\eta>0$ be assigned, as in the statement. By compactness of $X$, there exists $\varepsilon_0>0$ such that
    \begin{equation}\label{eq:PivotMetric}
    R_{g^{(s)}}(x,t)>R_g(x,0)-\frac12 \eta,\quad \forall s\in [0,1], x\in X\text{ and }t\in [0,\varepsilon_0].\end{equation}
    We define the metric $\hat g$ as follows:
    \[\hat g(x,t) = g(x,t) - \tau_{\delta,\varepsilon}(t) \left(\frac12 t^2 (\ddot h_0(x)+ 2C h_0(x))+Q_t(x)\right)\]
    for $0<\e<\e_0$ chosen small enough to fulfill property (1), depending on the assigned set $U$.
    Properties (2)-(3)-(4) follow immediately from the definition; property (5) can also be accommodated by possibly taking $\e$ even smaller, depending on $\eta$.
    
    Concerning (6), freeze a point $(x_0,t_0)$ with $0\leq t_0\leq \e$ (else there is nothing to prove), and let $s= \tau_{\delta,\varepsilon}(t_0)$. Then we have:
    \begin{multline*}
	2\|\hat g  - g^{(s)}\|_{C^2((x_0,t_0),g)} 
	\le \|(s-\tau_{\delta,\varepsilon}(t))( t^2 (\ddot h_0(x) + 2C          h_0(x))+2Q_t(x))\|_{C^2((x_0,t_0),g)} \\
 \le C |\log \delta|^{-1}.
    \end{multline*}
    Thus, by taking $\delta$ sufficiently small, it follows that 
    \[|R_{\hat g}(x,t) - R_{g^{(s)}}(x,t)|\le \frac12 \eta.\]
    We hence conclude, combining the previous inequality with \eqref{eq:PivotMetric}, that there holds indeed $R_{\hat g}(x,t)>R_g (x,0)-\eta$ as claimed.
\end{proof}

In particular, if the metric $g$ satisfies that $R_g>0$ and $H_g=0$ (resp. $H_g\ge 0$), we may locally deform it, inside a pre-assigned neighborhood $U$ of the base $F$, to a $C$-normal metric $\hat g$ such that $R_{\hat g}>0$ on $M$, $H_{\hat g}=0$ (resp. $H_{\hat g}\geq 0$) on $\partial M\setminus U$, and $|H_g|<\eta$ (resp. $H_g>-\eta$) on $(\partial M\setminus F)\cap U$, without changing its induced metric and the second fundamental form on $F$ . 

\section{Prescribing second fundamental form and warping factor}\label{sec:Prescribing}

Keeping in mind the output of the construction given in the previous section, we shall now assume to be given, on the manifold with corners $M=X\times\R^+$, a smooth metric $g$ that is $C$-normal and takes the form
\[g(x,t) = u(x,t)^2 dt^2 + h_0(x) - 2th_1(x) - Ct^2 h_0(x), \quad x\in X, t\ge 0\]
in an open neighborhood $U$ of $X$. For purely expository convenience let us assume that the function $u$ is extended to $X\times \R^+$.

\begin{proposition}\label{prop:deform_prescribe_II}
    Assume the setup above and that $R_g>0$ in $U$. Given a symmetric $(0,2)$-tensor $k$ on $X$ satisfying $\tr_{h_0}k\le \tr_{h_0}h_1$, there exists 
      $C_0=C_0(h_0, h_1,k,u)$ such that, for every $C\ge C_0$ there exist a metric $\hat g$ on $M$ and an open set $\emptyset\neq\hat U=\hat{U}(C)\subset U$ such that $\hat g=g$ in $M\setminus U$, \begin{equation}\label{eq:BarHanke_deformation1}
        \hat g= u(x,0)^2 dt^2 + h_0(x) - 2t k(x) -Ct^2 h_0(x)
    \end{equation}
    in $\hat U$, and satisfying:
    \[
    R_{\hat g}>0 \ \text{in} \ U, \text{and} \ 
        |H_{\hat g}-H_g|< \eta \ \text{on} \ (\partial M\setminus F)\cap U.
    \]
\end{proposition}

We recall that $\partial M$ denotes the topological boundary of $M$, and $F$ is the notation we employ for $X\times\left\{0\right\}$.
In order to prove the previous statement, achieving the desired deformation, we will take the same cutoff function as in \cite[Lemma 25]{BarHan20}.

\begin{lemma}\label{lem:BarHanke_testfunction2}
    There exists a constant $c_0>0$ such that for each $\e\in (0,\frac12)$, there exists a smooth function $\chi_{\e}: [0,\infty)\to \R$ such that:
    \begin{enumerate}[(1)]
        \item $\chi_{\e}(t)=t$ for $0\leq t\leq \e/20$, $\chi_{\e}(t)=0$ for $t\ge \sqrt{\e}$, and $0\le \chi_{\e}(t)\le \e/2$ for all $t$;
        \item $|\chi_{\e}'(t)|\le c_0$;
        \item $-2/\e\le \chi_{\e}''(t)\le 0$ for all $t\in [0,\e]$ and $|\chi_{\e}''(t)|\le c_0$ for all $t\in [\e, \sqrt{\e}]$.
    \end{enumerate}
\end{lemma}

\begin{proof}[Proof of Proposition \ref{prop:deform_prescribe_II}]

Employing the test function $\chi_{\e}(t)$ constructed in Lemma \ref{lem:BarHanke_testfunction2}, we first define an auxiliary metric $\tilde g$ as
    \[\tilde g(x,t) = u(x,t)^2 dt^2 + h_0(x) -2t h_1(x) + 2\chi_{\e}(t)(h_1(x) - k(x)) - Ct^2 h_0(x),\]
    when $t\le \sqrt{\e}$, and $\tilde g=g$ when $t\ge \sqrt{\e}$; here $\e$ is chosen small enough to ensure that this interpolation occurs inside the assigned open set $U$. It is clear that $\tilde g$ is $C$-normal  and has the desired form as in \eqref{eq:BarHanke_deformation1}, except for the constancy of the warping factor (which will be arranged later), in the neighborhood of $F$ determined by the condition $0\leq t\leq \e/20$. We verify that $\tilde g$ satisfies the conclusions on the scalar curvature and the mean curvature.

Indeed, writing $\tilde g = u(x,t)^2 dt^2 + h_t(x)$, we first note that the smooth metrics on $X$ given by $\{h_t\}_{t\in [0,\sqrt{\e}]}$ are uniformly close in the $C^2$ topology to $h_0$ as long as one imposes  $\sqrt{\e}C\leq 1$, which we henceforth assume throughout this proof. For later reference, let us spell out the $C^0$ estimate: one has for all $t\in [0,\sqrt{\e}]$
\begin{equation}\label{eq:hth0}
\|h_0-h_t\|_{h_0}\leq 2t(\|h_1\|_{h_0}+\|h_0\|_{h_0})+2|\chi_{\e}(t)|\|h_1-k\|_{h_0}
\end{equation}
which can uniformly be bound from above by
$4\sqrt{\e}(\|h_1\|_{h_0}+\|h_0\|_{h_0}+\|k\|_{h_0}).$

As $\e\to 0$, we have that $R_{h_t} = O(1)$ and $\frac{2\Delta_{h_t}u}{u}=O(1)$; here and throughout this proof $O(1)$ represents a constant that may change line from line, but is uniformly bound in $\e$, independently of $C$; similarly $O(\e)$ has the meaning of $\e\cdot O(1)$.

     Moreover, we have $\dot h_t = -2h_1 + 2\chi_{\e}'(t)(h_1-k)-2Cth_0$. Hence, applying a Linear Algebra estimate like that in \cite[Lemma 24]{BarHan20}, as long as $\e$ is sufficiently small so that $\|h_0-h_t\|_{h_0}<\frac12$ we have for all $t\in [0,\sqrt{\e}]$
    \begin{align*}
        |\tr_{h_t} \dot h_t|\le |\tr_{h_0}\dot h_t| + 2\|h_t- h_0\|_{h_0} \cdot \|\dot h_t\|_{h_0}=O(1)
    \end{align*}
    where the last bound holds thanks to the usual constraint $\sqrt{\e}C\leq 1$ and the uniform bound $|\chi_{\e}'(t)|\le c_0$ from Lemma \ref{lem:BarHanke_testfunction}.
    
    Next, we have that $\ddot h_t = 2\chi_{\e}''(t)(h_1-k) - 2Ch_0$. Thus, in a much similar fashion, we have that
    \[\tr_{h_t}(\chi_{\e}''(t)(h_1-k)) \le \chi_{\e}''(t)(\tr_{h_0}h_1 - \tr_{h_0}k ) + 2|\chi_{\e}''(t)|\|h_t-h_0\|_{h_0}\cdot \|h_1-k\|_{h_0},\]
    which we estimate distinguishing two cases.
    When $t\in [0,\e]$, note that $\chi_{\e}''\le 0$ and $\tr_{h_0} h_1\ge \tr_{h_0} k$. Thus, we can drop the first summand and have that
    \begin{align*}\tr_{h_t}(\chi_{\e}''(t)(h_1-k))&\le 2|\chi_{\e}''(t)| \|h_0-h_t\|_{h_0} \cdot \|h_1-k\|_{h_0}\\
    &\le \frac{4}{\e}\cdot 4\e (\|h_1\|_{h_0}+\|h_0\|_{h_0}+\|k\|_{h_0}) \|h_1-k\|_{h_0} =O(1)
    \end{align*}
    building upon \eqref{eq:hth0}.
    When instead $t\in [\e,\sqrt{\e}]$, we have that $|\chi_{\e}''(t)|\le c_0$, so $|\tr_{h_t}(\chi_{\e}''(t)(h_1-k))|=O(1)$. Hence for all $t\in [0,\sqrt{\e}]$, we have that
    $\tr_{h_t} (\chi_{\e}''(t)(h_1-k))= O(1)$.
    Note that for $t\in [0,\sqrt{\e}]$, we analogously have 
    \[|\tr_{h_t} (2Ch_0) - \tr_{h_0}(2Ch_0)|\le 2C\|h_t-h_0\|_{h_0}\|h_0\|_{h_0} = O(1).\]
    Thus, exploiting the leading contribution $\tr_{h_0} (2Ch_0)=2Cn$ (where, let us recall $n$ is the dimension of $X$) we conclude that for all $t\in [0,\sqrt{\e}]$,
    \begin{equation}
    \label{eq:ScalTildeg}
    R_{\tilde g} \ge \frac{2Cn}{u^2} - O(1)>0 
    \end{equation}
    provided we simply choose $C$ large enough; indeed we shall choose $C$ first and then $\e$ sufficiently small based on the aforementioned constraint $\sqrt{\e}C\leq 1$. (However, such choices will actually be made later for we still need to modify the warping factor.)

    Next, we verify that $\tilde g$ satisfies the mean curvature property along the cylindrical boundary, that is $\partial X\times \R^+$. This is easily seen. By item (4) of Lemma \ref{lem:Formulae}, we have that $H_{\tilde g} = u^{-1} (H_h u + \partial_\nu u)$. Since, as noted above, $h_t\to h_0$ as $t\to 0$ in the $C^2$ topology, and similarly $u_t\to u_0$ uniformly in the $C^1$ topology, we have that $H_{\tilde g}(x,t)\to H_g(x,0)$ uniformly for $x\in X, t\in [0,\sqrt{\e}]$ as $\e\to 0$ (plus of course by design the metric coincide as soon as $t\geq \sqrt{\e}$).

    To complete the construction we still need to modify the warping factor. With that goal in mind, we construct $\hat g$ as follows. Take a positive smooth function $\lambda_{\e}:[0,\infty)\to \R$ such that
    \[\lambda_{\e}(t) = \begin{cases}
        0 \quad \text{if} \ t\le \e^2 \\ 
        t \quad \text{if} \ t\ge \frac{\e}{2}
    \end{cases},\quad 0\le \lambda_{\e}'(t)\le 2.\]
    Then define $\hat u(x,t) = u(x,\lambda_{\e}(t))$, and in turn
    \[\hat g = \hat u(x,t)^2 dt^2 + h_0(x) - 2th_1(x) +2\chi_{\e}(t) (h_1(x)-k(x)) -Ct^2h_0(x).\] 
    It is apparent that $\hat g$ satisfies \eqref{eq:BarHanke_deformation1} in $\hat U:=\left\{(x,t)\in X\times \R^+ \ : \ 0\leq t<\e^2 \ \right\}$. To verify the conditions on the scalar curvature and mean curvature of $\hat g$, we compare $\hat g$ and $\tilde g$ when $t\le \e$; indeed for $t\geq \e/2$ one has $\hat{g}=\tilde{g}$ and so we can just rely on the preceding part of the proof.

 Regarded both $\hat u$ and $u$ as differentiable mappings from $t\in \R$ to the Banach space $C^2(X)$, namely considering $\hat{u}, u\in C^1([0,\e];C^2(X))$ we have that $\|\hat u- u\|_{C^0(t; C^{2}(X))}=O(\e)$, and $\|\hat u - u\|_{C^1(t;C^{2}(X))}=O(1)$ (i.\,e. the difference is bounded uniformly in $\e$). Thus, by inspecting the various summands of \eqref{eq:warped.BarHanke.scalar.curvature} it is easily checked that $R_{\hat g}\ge R_{\tilde g} - O(1)$; in particular, let us remark that here one has to deal with the term
 \[
\left(\frac{1}{\hat{u}^2}-\frac{1}{u^2}\right)\tr_{h_t}{\ddot h_t}=\frac{(u-\hat{u})(u+\hat{u})}{(u\hat{u})^2}\tr_{h_t}{\ddot h_t}
 \]
 whose boundedness (in fact: smallness) builds upon the fact that $\e C=\sqrt{\e}(\sqrt {\e} C)\leq \sqrt{\e}$ by virtue of the usual constraint.

 As a result, it follows from the aforementioned inequality $R_{\hat g}\ge R_{\tilde g} - O(1)$ and the preceding bound \eqref{eq:ScalTildeg} that taking $C$ sufficiently large depending on the input data only, we can arrange $R_{\hat g}>0$. The argument for the mean curvature of the cylindrical boundary is similar in spirit but even simpler in practice: just based on the fact that $\|\hat u- u\|_{C^0(t; C^{2}(X))}=O(\e)$ we conclude that $H_{\hat g}(x,0)$ is $\e$-close to $H_{\tilde g}(x,0)$ and so for suitably small $\e$, also depending on the assigned $\eta>0$, we can conclude the proof.
    \end{proof}

    So, in short, the proposition above allows to \emph{push down} the mean curvature (by an arbitrary, not necessarily small amount) while keeping the ambient scalar curvature positive; in the meantime one can also accomodate the convenient condition of making the warping factor independent of $t$ (namely: independently of the slice in the local foliation) near the interface. Such features are very convenient when it comes to proving our main result below.  We also wish to add a more technical comment on the preceding construction.

\begin{remark}\label{rem:Miracle}
Let us consider the outcome of the construction above in the annular domain $\hat U$ where $0\leq t<\e^2$. In that region, if $k=0$ there holds $\hat g=u^2(x,0)dt^2+(1-Ct^2)h_0$. In particular, set $h=(1-Ct^2)h_0$ the mean curvature of the cylindrical boundary at a point of coordinates $(x,t)$ equals
\[
H_{h}(x,t)+\partial_{\nu_h}\log u(x,0).
\]
But, just by scaling, $\nu_h(x,t)=\nu_h(x,0)/\sqrt{1-Ct^2}$ and similarly $H_h(x,t)=H_h(x,0)/\sqrt{1-Ct^2}$. As a result, if $H_g=0$, or respectively $H_g\geq 0$, along $\partial X\times \left\{0\right\}$ then the same conclusion shall, respectively, hold true for $\hat{g}$ in the whole domain $\hat U$ in question.
\end{remark}

  \section{Statement and proof of the main result}\label{sec:Desing}

    In this section we wrap up things to prove our main result:

\begin{theorem}\label{thm:Desing}
    Let $(M_{-},g_{-})$ and $(M_+,g_+)$ be (smooth) compact Riemannian manifolds with corners, both having dimension $n+1\geq 3$; assume $F_{-}\subset M_{-}$ (respectively: $F_{+}\subset M_{+}$) are cylindrical faces (as defined in the first paragraph of Section \ref{sec:Cnormal}) and there exists $\phi: M_{-}\to M_{+}$ giving an isometry between $[g_{-}]_{|F_{-}}$ and $[g_{+}]_{|F_{+}}$. Let $M:=M_{-}\sqcup_{\phi}M_{+}$, with its natural atlas of manifold with corners and let $\pi_{\pm}: M_{\pm}\to M$ be the corresponding projection maps (for either consistent choice of signs). We shall then introduce the following notation:
    \begin{itemize}
    \item $X$ is the codimension-one submanifold that is the common image of $F_{+}, F_{-}$ in $M$;
    \item $Y_{\pm}\subset\partial M_{\pm}$ is the disjoint union of all faces of $M_{\pm}$ having non-empty intersection with $F_{\pm}$, and $Y$ the disjoint union of all faces of $M$ having non-empty intersection with $X$.
    \end{itemize}

    Suppose that $R_{g_\pm}>0$ on $M_{\pm}$, that $Y_{\pm}\subset \partial M_\pm$ are mean-convex, that $Y_{\pm}$ meet ${F_{\pm}}$ at a right angle, and in addition there holds for the mean curvature of the isometric faces
    \begin{multline}\label{eq:JumpCond}
H_{g_{-},F_{-}}\ge f(x), \quad H_{g_{+},F_{+}}\ge -f(x), \\
\text{where \underline{either}} \ f(x)\geq 0 \ (\forall x\in X)  \ \ \ \text{\underline{or}} \ f(x)\leq 0 \ (\forall x\in X).
\end{multline}
      Given an open neighborhood $U$ of $X$ in $M$, such that $U\cap \partial M\subset Y$ (thus disjoint from $\text{sing}(M)$), there exist a Riemannian metric $g$ on $M$ and an open set $\hat U\subset U$ such that the restriction of $g$ to $M\setminus \hat U$ satisfies $\pi^{\ast}_{\pm} g = g_\pm$ on $M_\pm$, and in $U$ the following two properties hold:
\begin{enumerate}[(1)]
        \item $(M,g)$ has positive scalar curvature;
        \item $(M,g)$ has mean-convex boundary, and in fact minimal boundary if the same is assumed to be true for $Y_{-}, Y_{+}$ respectively in $(M_{-},g_{-})$ and $(M_+,g_+)$.
    \end{enumerate}
\end{theorem}

\begin{remark}
Some comments on the assumptions are appropriate:
\begin{itemize}
\item  condition \eqref{eq:JumpCond} coincides (given the different sign convention) with the jump condition $(\textbf{H})$ in \cite{Mia02}, together with the additional technical requirement that at least one of the two functions $H_{g_{+},F_{+}}, H_{g_{-},F_{-}}$ does not change sign;
\item we stress the construction we present is  local near the given interface, so the singularities of $M$ away from $X$ do not play any role; that said, an important special case occurs when
\[
\text{sing}(M_{-})\setminus F_{-}=\text{sing}(M_{+})\setminus F_{+}=\emptyset
\]
for in that case the output of the theorem is a smooth compact manifold with positive scalar curvature and mean-convex boundary (or minimal boundary under the same assumption on the input data); see Corollary \ref{cor:HalfSphere} for an example of that sort.
\end{itemize}
\end{remark}

\begin{remark}\label{Rem:GenVersionMainThm}
With rather minor modifications, we can in fact prove a more general version of the statement above, enforcing in $U$ a lower bound on the scalar curvature of the form $R_g>R_0$ for any assigned $R_0\in\R$, provided the same is required on the metrics $g_{-}, g_{+}$ (together with all other standing assumptions). The necessary changes to our arguments are outlined in Remark \ref{rem:ExtGenBounds}.
\end{remark}

Like we anticipated in the introduction, the proof of Theorem \ref{thm:Desing} consists of two steps. Step 1 is, at least formally, a preparatory step, which concerns how to best ``interpolate'' between two manifolds with corners having one isometric face of cylindrical type, in the sense we explained at the beginning of Section \ref{sec:Cnormal}; the main issue one needs to deal with is the jump at the level of the warping factor. 
  In the proof of this fact we will crucially employ the function $\varphi_{\e}$ that is the object of Lemma \ref{lem:test.function.for.gluing}; its construction, albeit elementary, is fairly technical, which is why we decided to postpone it to Appendix \ref{app:GluingFunction}. 

As a preliminary note, we further remark (recalling Section \ref{sec:NormalForm})
that there exists a smooth ($C^\infty$) function $t:M\to \R$ that allows to write both $g_{-}$ and $g_{+}$ in block form, in the sense of Proposition \ref{prop:NormalForm}; this relies on the very definition of the atlas on $M$ itself. In particular, $\{t\le 0\}$ (resp. $\{t\ge 0\}$) is identified with a small neighborhood of $F_{-}$ in $M_{-}$ (resp. of $F_{+}$ in $M_{+}$). We emphasize that, in the manifold $M$, the vector $\partial_t$ restricted to $X$ points out of (the projection of) $M_{-}$, into (the projection of)  $M_{+}$.

\begin{proof}

\textbf{Step 1:} under the assumptions of the statement, we will reach here the following preliminary conclusions:

\emph{given $\delta>0$ there exist a smooth Riemannian metric $\hat g$ on $M$ and an open set $\hat{U}\subset U$ such that the restriction of $g$ to $M\setminus \hat U$ satisfies $\pi^{\ast}_{\pm}\hat g = g_\pm$ on $M_\pm$, and in addition:
  \begin{enumerate}
        \item [(1)'] $R_{\hat g}>0$ in $U$, and
        \item [(2)'] $H_{\hat g}>-\delta$ on $\partial M \cap U$.
    \end{enumerate}
    Furthermore, in the special case when that $\partial M_\pm \setminus F_{\pm}$ are minimal, then conclusion (2)' above can be upgraded to the two-sided bound 
    $(2)'_{m}:$ $|H_{\hat g}|<\delta$ on $\partial M \cap U$.}

Without loss of generality, let us assume that $f\ge 0$ (the treatment of the case $f\leq 0$ only differs by the former in one point, which we will highlight at due course, the logical structure of the argument being identical). Applying Proposition \ref{prop:C_normal_deformation} and then Proposition \ref{prop:deform_prescribe_II}, we may assume, without loss of generality, that
    \[g_\pm = u_\pm (x)^2 dt^2 + h_0(x) - \frac{2}{n}t u_\pm (x) f(x) h_0(x) - Ct^2 h_0(x),\]
    when $|t|$ is sufficiently small; note that, because of such preliminary deformations the mean curvature of the cylindrical boundary $Y$ will suffer of a local error term, namely in the mean-convex (respectively: minimal) case
    \begin{equation}\label{eq:CylErrTerm}
H_{g_{\pm}}>-\delta/2 \ \ (\text{respectively}: \ |H_{g_{\pm}}|<\delta/2) \ \ \text{if} \ \ |t|<\e_0=\e_0(\delta),
    \end{equation}
    while we still have $H_{g_{\pm}}\geq 0$ outside of such a strip (or $H_{g_{\pm}}\equiv 0$ in the special case when the input metric are assumed to have minimal boundary along $Y_{\pm}$). Here it is convenient to restrict to the range $\delta\in (0,\delta_0]$, for some $\delta_0>0$ fixed once and for all, and  choose $\e_0(\delta_0)>0$ so small that $\left\{|t|\leq 4\e_0(\delta_0)\right\}\subset U$; we note that it follows from the proofs of Proposition \ref{prop:C_normal_deformation} and Proposition \ref{prop:deform_prescribe_II} that one can actually require $\delta\mapsto \e_0(\delta)$ to be non-decreasing so that in particular there holds $\left\{|t|\leq 4\e_0(\delta)\right\}\subset U$ for any $\delta\in (0,\delta_0]$.
    
    Observe that by suitably rescaling (cf. Remark \ref{rem:Scaling}), we may also assume, without loss of generality, that
    $u_{-}(x)< u_{+}(x)$ for all $x\in X$; if $f\leq 0$ we shall arrange for the reverse inequality instead.

    For any $\e>0$ sufficiently small,  $\e\in (0,\e_0(\delta)]$ subject to the general constraint $C\e\leq 1$ (that is henceforth enforced throughout the proof) and $t\in [0,\varepsilon]$, define $u_t:X\to \R$ by
    \[\log u_t(x) = \left(1-\frac{\varphi_{\e}(t)}{\varepsilon}\right) \log u_{-} + \frac{\varphi_{\e}(t)}{\varepsilon} \log u_{+},\]
    where $\varphi_{\e}(t)$ is the interpolation function constructed in Lemma \ref{lem:test.function.for.gluing} (cf. Appendix \ref{app:GluingFunction}), and
    \[\hat g (x,t)= \begin{cases}
        g_{-}(x,t)\quad & t\le 0;\\
        u_t(x)^2 dt^2+h_t  & 0\le t\le \varepsilon;\\
        g_{+}(x,t), & t\ge \varepsilon,
    \end{cases}\]
where for $t\in [0,\varepsilon]$ we set
\begin{equation}\label{eq:TransverseMetric}
h_t(x)=h_0(x) - \frac{2}{n} t u_t(x) f(x)h_0(x) - Ct^2 h_0(x).
\end{equation} 
    By definition (relying upon the defining properties of $\varphi_{\e}$), $\hat g$ is a smooth Riemannian metric on $M$, and conclusions (1)' and (2)' or $(2)'_m$ in the minimal case patently hold when $t\le 0$ or $t\ge \varepsilon$. We now check that it satisfies (1)' and (2)', or $(2)'_m$ when $t\in [0,\varepsilon]$; 
    we will check the latter requirement first. Observe that for all $t\in [0,\e]$ $u_t(x)$ is uniformly bounded in $C^2(X)$ and has a uniform \emph{positive} lower bound; note further that $h_t\to h_0$ as metrics on $X$ in the $C^2$ topology. 

    That being said, let us consider property $(2)'_m$. We start by noting that, keeping in mind \eqref{eq:CylErrTerm} and Lemma \ref{lem:Formulae}, along the interface there holds
    \[|H_{h_0}(x) + \partial_{\nu} (\log u_\pm)|< \frac{\delta}{2},\]
    where $\nu$ is - as usual - the outward unit normal vector on $\partial X$ in metric $h_0$. If we then introduce the auxiliary metric 
    \[
    \tilde g(x,t) = u_t(x)^2 dt^2+ h_0(x)\]
    we have, again for the cylindrical boundary along the interface, i.\,e. along $X\cap Y$
    \begin{align*}
     H_{\tilde g}(x,t) & = H_{h_0}(x)+\left(1-\frac{\varphi_{\e}(t)}{\varepsilon}\right) \partial_{\nu} \log u_{-} + \frac{\varphi_{\e}(t)}{\varepsilon} \partial_{\nu} \log u_{+} \\
     & =\left(1-\frac{\varphi_{\e}(t)}{\varepsilon}\right)(H_{h_0}+\partial_{\nu} \log u_{-})+\frac{\varphi_{\e}(t)}{\varepsilon}(H_{h_0}+ \partial_{\nu} \log u_{+}),
    \end{align*}
    whence convexity implies that $|H_{\tilde g}|<\delta/2$ on the strip defined by $0\leq t\leq\varepsilon$. Since the metrics $h_t$ and $h_0$ are $C^2$ close on $X$, we have that
    $|H_{\hat g}(x,t)|<\delta$ along that same strip
    provided $\varepsilon$ is taken sufficiently small depending further on $\delta$. The verification of (2)' follows along the same line, modulo working (both at the level of assumption and conclusion) with one-sided inequalities, i.\,e. lower bounds only.

    Let us then discuss the validity of part (1)' of the statement instead. Appealing again to the convergence $h_t\to h_0$ in $C^2(X)$, we have that 
    \[R_{h_t} = O(1),\quad\frac{\Delta_{h_t} u_t}{u_t}= O(1), \quad \text{uniformly as }\varepsilon\to 0.\]
    Note that in this proof, as in that of Proposition \ref{prop:deform_prescribe_II}, we use $O(1)$ to denote a bounded quantity that may change from line to line, but only depends on $u_\pm, f, h_0$ (but not on the constant $C$, which plays a specific role as clarified in Definition \ref{def:Cnormal}); similarly for $O(\e)=\e\cdot  O(1)$.

    To proceed further, keeping in mind \eqref{eq:warped.BarHanke.scalar.curvature}, we differentiate the map $t\mapsto h_t$ and find
    \begin{equation}\label{eq:FirstOrderDerH}
    \frac{1}{u_t} \dot h_t = - \frac{2}{n}\left(1 + t\frac{d}{dt}(\log u_t)\right)f(x) h_0(x) - \frac{2Ct}{u_t} h_0. 
    \end{equation}

    Hence, note that \[
    t\frac{d}{dt}(\log u_t) = \frac{t\varphi_{\e}'(t)}{\varepsilon}(\log u_{+}-\log u_{-})
    \] and so, by item (5) of Lemma \ref{lem:test.function.for.gluing}, we have that 
    this term is $O(\varepsilon)$, while the other two in \eqref{eq:FirstOrderDerH} are $O(1)$. It then follows at once that the ``first-order terms'' \[
    \frac{\dot{u_t}} {u^3}\tr_{h_t}\dot h_t, \quad \frac{1}{u^2}|\dot h_t|^2, \quad \frac{1}{4u^2}\left(\tr_{h_t}\dot h_t\right)^2
    \]are all $O(1)$ in the sense above.  
  Next, differentiating \eqref{eq:FirstOrderDerH} once again there holds
    \begin{align*}
        -\frac{d}{dt}\left(\frac{1}{u_t}\dot h_t\right) &= \frac{2}{n} \frac{d}{dt}\left(t\left(\frac{d}{dt}\log u_t\right)\right)f(x)h_0(x)
         + \frac{2C}{u_t}h_0 - 2C \frac{t}{u_t}\left(\frac{d}{dt} (\log u_t)h_0(x)\right)\\
            &= \frac{2}{n}(\log u_{+} - \log u_{-})f(x)h_0(x)\cdot \frac{1}{\varepsilon} \left(t\varphi_{\e}'(t)\right)' + \frac{2C}{u_t} h_0 - 2C \frac{t}{u_t}\left(\frac{d}{dt}(\log u_t)h_0(x)\right).
    \end{align*}
    Recall that, because of \eqref{eq:warped.BarHanke.scalar.curvature}, we need to estimate
    \[
- \frac{1}{u^2} \tr_{h_t} \ddot h_t=-\frac{1}{u}\tr_{h_t}\left(\frac{d}{dt}\left(\frac{1}{u_t}\dot h_t\right)+\frac{\dot u_t}{u^2_t}\dot h_t\right)
        \]
        and so it is clear, since by the preceding discussion we have control on the first-order terms, that it is sufficient for our purposes to find a lower bound on the first summand on the right-hand side, which is indeed what we have computed above.
    As before (Section \ref{sec:Prescribing}), we have the estimate $\tr_{h_t}\left(\frac{2C}{u_t}h_0\right)= \frac{2nC}{u_t}+O(1)$, while 
    \[
    \tr_{h_t}\left(2C \frac{t}{u_t}\frac{d}{dt}(\log u_t)h_0\right)= 2C\cdot O(\varepsilon)=O(1)\] where we have exploited twice our requirement that $C\e\leq 1$. The only remaining term we would need to estimate is 
    \[\tr_{h_t}\left(\frac{2}{n}(\log u_{+} - \log u_{-})f(x)h_0(x)\cdot \frac{1}{\varepsilon}(t\varphi_{\e}'(t))'\right).\]
    Towards that goal, we have that
    \[ 2\log\left(\frac{u_{+}}{u_{-}}\right)f(x) \cdot \frac{1}{\varepsilon}(t\varphi_{\e}'(t))'
          \ge -4\log\frac{u_{+}}{u_{-}} f(x) 
          = O(1)
    \]
    which gives a lower bound on $\tr_{h_0} \left(\frac{2}{n}\log \frac{u_{+}} {u_{-}}f(x)h_0(x)\cdot \frac{1}{\varepsilon}(t\varphi_{\e}'(t))'\right)$. 
    Here we have used that $u_{+}>u_{-}$, $f(x)\ge 0$, and conclusion (4) of Lemma \ref{lem:test.function.for.gluing}. About the resulting summand, that is 
    \[
    \tr_{h_t}\left(\frac{2}{n}\log \frac{u_{+}}{u_{-}}f(x)h_0(x)\cdot \frac{1}{\varepsilon}(t\varphi_{\e}'(t))'\right) 
    - \tr_{h_0}\left(\frac{2}{n}\log \frac{u_{+}}{u_{-}}f(x)h_0(x)\cdot \frac{1}{\varepsilon}(t\varphi_{\e}'(t))'\right)
    \]
     we employ the usual pointwise algebraic inequality and claim its absolute value is bounded by
    \[
        2\|h_t - h_0\|_{h_0} \cdot \left\|\frac{2}{n}\log \frac{u_{+}}{u_{-}}f(x)h_0(x)\cdot \frac{1}{\varepsilon}(t\varphi_{\e}'(t))'\right\|_{h_0}
            \le O(1) \left|\frac{t}{\varepsilon} (t\varphi_{\e}'(t))'\right|
            = O(1).
    \]
    thanks to item (6) of Lemma \ref{lem:test.function.for.gluing}. Putting up these estimates together, we have that, when $t\in [0,\varepsilon]$, 
    \begin{equation}\label{eq:ScCurvBalance}
     R_{\hat g} \ge \frac{2nC}{u^2}-O(1) > 0,
     \end{equation}
    provided that $C>C^*(u_\pm, f, h_0)$ and, in turn, $\e<\min\left\{\e_{\ast}(u_\pm, f,h_0, \delta),\e_0(\delta)\right\}$. Thereby, the conclusion follows.

\
  
\textbf{Step 2:} let $\hat{g}$ be the smooth metric on $M$ that we just constructed, and let $t:M\to\R$ the ``signed distance'' function that we employed. It is appropriate to emphasize the dependence on $\delta$ so we will rather write $\hat{g}_{\delta}$ throughout; from now onward, we will conveniently write $\sigma$ in lieu of $\e_0(\delta_0)$ as introduced in Step 1.

At this stage, we will apply a conformal deformation to the metrics $\hat{g}_{\delta}$ (where, recall, $\delta\in(0,\delta_0]$) to obtain a continuous family of metrics with the desired features as soon as $\delta\in (0,\delta_1)$, for a suitably chosen positive $\delta_1<\delta_0$, with $\hat{U}=\left\{(x,t)\in U \ : \ |t|<2\sigma\right\}$.
To do so properly, we will employ (without renaming) a modification of the smooth cutoff function of Lemma \ref{lem:BarHanke_testfunction2}; in fact, all we need to require is that $\chi_{\mu}(x)=x$ for $x\leq \mu$, $\chi_{\mu}(x)=0$ for $x\geq 2\mu$ and be positive for $x\in (0,2\mu)$ under the sole additional constraint that $\chi_{\mu}(x)\leq x$ for any $x$. Let, further, $Y$ be as in the statement of the theorem, and let $Z$ be the union of all other faces of $M$; in particular $\partial M=Y\sqcup Z$ modulo the (possibly empty) singular set where they overlap.

Thanks to Step 1, we have that for any $\delta>0$ small enough there exists $\varrho=\varrho(\delta)>0$ with $\varrho(\delta)\to 0$ as $\delta\to 0$ such that the principal eigenvalue for the problem
\begin{equation}\label{eq:BVPforConf}
    \begin{cases}
    -\lapl_{\hat{g}_{\delta}} w+\frac{1}{c(n)}\chi_{\varrho(\delta)}(R_{\hat{g}_{\delta}}) w  =\lambda w & \text{on} \ M \\
   \hspace{4.3mm} \partial_{\nu_{\delta}} w +\frac{2}{c(n)} \chi_{\varrho(\delta)}(H_{\hat{g}_{\delta}}) w  =0  & \text{on} \ Y \\
   \hspace{36.4mm} \partial_{\nu_{\delta}} w  =0  & \text{on} \ Z 
    \end{cases}
    \end{equation}
 that - let us recall - has the variational characterization
\begin{equation}\label{eq:VarCharProblem}
\lambda_1=\inf_{w\neq 0}\frac{\int_M |\nabla_M w|^2+\frac{1}{c(n)}\chi_{\varrho(\delta)}(R_{\hat{g}_{\delta}}) w^2+\frac{2}{c(n)}\int_{Y} \chi_{\varrho(\delta)}(H_{\hat{g}_{\delta}}) w^2}{\int_{M}w^2}
\end{equation}
is \emph{strictly positive} for $\delta$ sufficiently small (cf. Appendix A in \cite{CarLi19}); note that, without loss of generality, we can (and shall) require in $\varrho(\delta)>\delta$ for all $\delta\in(0,\delta_0)$. Here and in the sequel of the proof it is set $c(n)=4(n-1)/(n-2)$. Then, it also (automatically from the fact that we have chosen the cutoff thresholds so that $\varrho(\delta)\to 0$ as $\delta\to 0$) follows that 
\begin{equation}\label{eq:Smallness}
 \|\chi_{\varrho(\delta)}(R_{\hat{g}_{\delta}})\|_{L^{\infty}(M)}\to 0,  \ \ \ \|\chi_{\varrho(\delta)}(H_{\hat{g}_{\delta}})\|_{L^\infty(Y)}\to 0 
\end{equation}
as $\delta\to 0$. By simply taking $w=1$ in \eqref{eq:VarCharProblem} it then follows that the principal eigenvalue in question satisfies
\begin{equation}\label{eq:Smallness2}
\lambda_1(\delta)\to 0, \ \ (\delta\to 0).
\end{equation}
Thus, given such smallness conditions \eqref{eq:Smallness} and \eqref{eq:Smallness2}, we have that Moser's Harnack inequality for \eqref{eq:BVPforConf} then implies that the corresponding first eigenfunction $u_\delta>0$ satisfies $\sup_U(u_{\delta})/\inf_U (u_{\delta})\to 1$ as $\delta\to 0$ and so, after suitably normalizing, $\|u_{\delta}-1\|_{C^0(-3\sigma\leq t\leq 3\sigma)}\to 0$ as $\delta\to 0$; in particular by standard elliptic regularity there holds in fact 
\begin{equation}\label{eq:C2}
\|u_{\delta}-1\|_{C^2(\sigma\leq |t|\leq 2\sigma)}\to 0, \ \ (\delta\to 0).
\end{equation}

Now we introduce the smooth function 
$\xi(p):=\overline{\xi}_{\sigma}(t(p))$ where $\overline{\xi}_{\sigma}\in C^{\infty}(\R;\R)$ is a ``bump function'' that equals 0 for $|s|\leq \sigma$, equals 1 for $|s|\geq 2\sigma$ and has a monotone transition inbetween. We then define (cf. Lemma 6.1 in \cite{LiMan21}) the conformally deformed metrics
\begin{equation}\label{eq:LocConfDef}
g_{\delta}=(\xi+(1-\xi)u_{\delta})^{\frac{4}{n-2}}\hat{g}_{\delta}.
\end{equation}
To check all claims in the statement of the theorem, we distinguish three regions:

\underline{for $|t|\geq 2\sigma$}: one has $\hat{g}_{\delta}=g_{\delta}=g_{\pm}$ and so all conclusions descend bijectively from the corresponding assumptions on the given metrics $g_{-}$ and $g_{+}$;

\

\underline{for $|t|\leq \sigma$}: one has $g_{\delta}=u_{\delta}^{\frac{4}{n-2}}\hat{g}_{\delta}$ and so, since (see e.\,g. \cite[Appendix B]{CarLi19}) in view of \eqref{eq:BVPforConf} one has
\[
R_{g_{\delta}}=u_{\delta}^{-\frac{n+2}{n-2}}[(R_{\hat{g}_{\delta}}-\chi_{\varrho(\delta)}(R_{\hat{g}_{\delta}}))+c(n)\lambda_1 u_{\delta}]
\]
we rely  and the positivity of the principal eigenvalue to ensure that $R_{g_{\delta}}>0$ for small enough $\delta$;

\

\underline{for $\sigma\leq |t|\leq 2\sigma$}: here we shall rather note that (because of \eqref{eq:C2}) the conformal factor in \eqref{eq:LocConfDef} converges uniformly in $C^2$ to $1$ in the given (compact) annular region, and so for any $\delta$ small enough the scalar curvature of the metric is positive.

In the regime $0\leq |t|\leq 2\sigma$, so for the last two cases above, however we still need to discuss the mean curvature of the boundary of $M$. Building again on the standard formulae for conformal change there holds (set $\phi:=\xi+(1-\xi)u_{\delta}$ for notational convenience)
\[
H_{g_{\delta}}=\phi^{-\frac{n}{n-2}}\left(H_{\hat{g}_{\delta}}\phi+\frac{c(n)}{2}\partial_{\nu_{\delta}}\phi\right)
\]
so that $H_{g_{\delta}}$ equals, modulo multiplication by a positive function,
\begin{equation}\label{eq:ConvexCombMC}
\xi H_{\hat{g}_{\delta}}+(1-\xi)(H_{\hat{g}_{\delta}}-\chi_{\varrho(\delta)}(H_{\hat{g}_{\delta}}))u_{\delta}
\end{equation}
which is thus the sum (in fact: convex combination) of two non-negative terms in the annular region, with just the second summand (that is patently non-negative) when $0\leq |t|\leq\sigma$. This completes the proof in the mean-convex case; in the minimal case it suffices to note that, following the same argument, the condition $\delta<\varrho_{\delta}$ implies that both summands in question are actually zero. 
\end{proof}

\begin{remark}\label{rem:WarpFactNoMatch}(The warping factors cannot possibly match)
 We explicitly note that the first step in the proof above, that is ``building a connecting bridge'' by interpolation, is in general \emph{unavoidable}. For indeed, since one that the faces to be attached be isometric, by virtue of Lemma \ref{lem:Formulae} it can be $u^{+}=u^{-}$ only if $H_{g_+, Y_{+}}=H_{g_{-},Y_{-}}$ at each point of the interface in question, namely only in the very special case when the mean curvature of the cylindrical boundary is the same, when measured in metric $g_{+}, g_{-}$, along the boundary of the interface.
\end{remark}

\begin{remark}\label{rem:Interp2sidedBounds}
It is apparent from the argument above that, if one could engineer an interpolation function $\varphi_{\e}:[0,\varepsilon]\to \R$ satisfying items (1), (2), (3), (4), (5) in Lemma \ref{lem:test.function.for.gluing} and, in lieu of (6), a stronger upper bound of the form
\begin{equation}\label{eq:ImpossibleUpperBound}
(t\varphi_{\e}'(t))'\le \mu \varepsilon, \ \ \forall t\in [0,\varepsilon]
\end{equation}
for some given $\mu>0$ (independent of $\varepsilon$) then the conclusion of Step 1 would hold true under the sole assumption $(\textbf{H})$, i.\,e. irrespective of any sign assumption on the separating function $f$. However, we note that the requirement \eqref{eq:ImpossibleUpperBound} is in fact incompatible already with the pair $(2), (3)$. Indeed, integrating \eqref{eq:ImpossibleUpperBound} twice gives - because of $(2)$ - a bound of the form $\varphi_{\varepsilon}(\varepsilon)\leq \mu\varepsilon^2$, which contradicts the independent assumption that $\varphi_{\varepsilon}(\varepsilon)=\varepsilon$ as soon as $\varepsilon<1/\mu$.
\end{remark}

Most importantly, we wish to outline here how to handle general lower scalar curvature bounds.

\begin{remark}\label{rem:ExtGenBounds}
It is possible to extend our main result, Theorem \ref{thm:Desing}, requiring scalar curvature lower bounds of the general form $R_g > R_0$ provided the same is assumed on the input metrics $g_{-}, g_{+}$ (together with all other standing assumptions). This is quite clear from the form of equation \eqref{eq:ScCurvBalance} for Step 1 (and analogously for the two preliminary steps carried through in the previous sections), as one is only required to choose $C>C^*(u_\pm, f, h_0, R_0)$ and, in turn, $\e<\min\left\{\e_{\ast}(u_\pm, f,h_0, \delta),\e_0(\delta)\right\}$. 
In Step 2, the only changes we need to make are to ``shift'' the cutoff function $\chi_{\mu}$ to $\chi^{R_0}_{\mu}$ satisfying $\chi^{R_0}_{\mu}(x)=x$ if $x\leq R_0+\mu$, $\chi^{R_0}_{\mu}(x)=R_0$ if $x\geq R_0+2\mu$ under the usual constraint that $\chi^{R_0}_{\mu}(x)\leq x$ for all $x$,  and to solve for the conformal factor considering the modified boundary value problem
\begin{equation}\label{eq:BVPforConfGen}
    \begin{cases}
    -\lapl_{\hat{g}_{\delta}} w+\frac{1}{c(n)}(\chi^{R_0}_{\varrho(\delta)}(R_{\hat{g}_{\delta}})-R_0) w   =\lambda w & \text{on} \ M \\
 \hspace{17.4mm}  \partial_{\nu_{\delta}} w +\frac{2}{c(n)} \chi_{\varrho(\delta)}(H_{\hat{g}_{\delta}}) w  =0  & \text{on} \ Y \\ \hspace{49.2mm} \partial_{\nu_{\delta}} w  =0  & \text{on} \ Z.
    \end{cases}
    \end{equation}
\end{remark}

\section{Some direct applications}\label{sec:Applications}

As anticipated in the introduction, we collect here some applications of the main theorem above, and of its method of proof. 
To begin with, here is a simple special instance of our smoothing result.

\begin{corollary}\label{cor:HalfSphere}
Let $n+1\geq 3$, let $M_{+}, M_{-}$ be both diffeomorphic, in the category of manifold with corners to
\[
\mathbb{D}^{n+1}_{+}:=\mathbb{D}^{n+1}\cap \left\{x^{n+1}\geq 0\right\}, \ \ \text{for} \ \ \mathbb{D}^{n+1}=\{(x^1,\ldots x^{n+1}) \ : \ \|x\|^2\leq 1\}
\]
let $F_{+}, F_{-}$ denote the faces associated to $\mathbb{D}^{n}$ by the diffeomorphism, and $Y_{+}, Y_{-}$ the other faces respectively in $M_{+}, M_{-}$. Assume we are given Riemannian metrics $g_{+}, g_{-}$ on $M_{+}, M_{-}$ that make $F_{+}$ isometric to $F_{-}$ and such that adjacent faces meet at a right angle, the scalar curvature of both $(M_{+}, g_{+})$ and $(M_{-}, g_{-})$ is positive, and the faces are all mean-convex. Then there exists a Riemannian metric on the closed Euclidean ball $\mathbb{D}^{n+1}$ that is equal to the given ones outside any pre-assigned neighborhood of the interface, has everywhere positive scalar curvature and mean-convex boundary. Furthemore, the boundary can be made minimal in the special case when the same is required for $Y_{\pm}$ in $(M_{\pm}, g_{\pm})$ for any consistent choice of signs.
\end{corollary}

We proceed by proving that the relations of weak PSC min-concordance and weak PSC mc-concordance in \cite{CarLi23} are indeed transitive, therefore equivalence relations as claimed.

\begin{corollary}\label{eq:Transitive} 
Let $X$ be a compact manifold with boundary and let us assume that there exist PSC Riemannian metrics $g_{0,1}$ and $g_{1,2}$ on $X\times [0,1]$ such that, in both cases, $\partial X\times [0,1]$ is minimal (respectively: mean-convex), and the slices $X\times \left\{0\right\}$ and $X\times\left\{1\right\}$ are free boundary minimal surfaces; furthermore, $g_{0,1}$ restricts to $h_0$ on $X\times \left\{0\right\}$, and to $h_1$ along $X\times \left\{1\right\}$; similarly $g_{1,2}$ restricts to $h_1$ on $X\times \left\{0\right\}$, and to $h_2$ along $X\times \left\{1\right\}$.

Then there exists a PSC Riemannian metric $g_{0,2}$ on $X\times [0,1]$ that makes $\partial X\times [0,1]$ minimal (respectively: mean-convex), and both $X\times \left\{0\right\}$ and $X\times\left\{1\right\}$ free boundary minimal surfaces, and in addition $g_{0,2}$
   restricts to $h_0$ on $X\times \left\{0\right\}$, and to $h_2$ along $X\times \left\{1\right\}.$
\end{corollary}

\begin{proof}
This is a straightforward application of Theorem \ref{thm:Desing}; we easily obtain a metric on $X\times [0,2]$ with the desired properties, after which it just suffices to reparametrize the time interval.
\end{proof}

To follow up on Remark 2.6 in \cite{CarLi23}, we then prove that the relation of weak PSC min-concordance (or, respectively, weak PSC mc-concordance) is the same as its strong counterpart, if one restricts a priori to the subclass $\met_{R>0,H=0}(X)$ (respectively: $\met_{R>0,H\geq 0}(X)$).

\begin{corollary}\label{cor:Equivalence}
Let $X$ be a compact manifold with boundary. If $h_{-1}, h_1\in \met_{R>0,H=0}(X)$  are weakly PSC min-concordant through a PSC metric $g$ on $M=X\times [-1,1]$ (thus: a metric making the cylidrical boundary minimal, and both faces free boundary minimal surfaces) then they are also (strongly) PSC min-concordant through a PSC metric $\overline{g}$  (thus: making the cylindrical boundary minimal, and being a Riemannian product in a neighborhood of both faces). 
\end{corollary}

The argument for justifying the preceding statement is quite striking in its simplicity, which ultimately builds on the fact that (for necessity) we developed methods to glue metrics dealing with mismatching warping factors.

\begin{proof} Firstly, consider on $X\times [1,2]$ the product metric $g_{(+1)}=dt^2+h_1$, and similarly $g_{(-1)}=dt^2+h_{-1}$ on $X\times [-2,-1]$. Since the interfaces $X\times\left\{\pm 1\right\}$ are both minimal on either side, we can then apply Theorem \ref{thm:Desing} to get the desired conclusion, up to reparametrizing in $t$.
\end{proof}

It is then clear that the very same strategy allows to prove the obvious analogue of Corollary \ref{cor:Equivalence} for weak PSC mc-concordance, when one assumes $h_{-1}, h_1\in \met_{R>0, H\geq 0}(X)$, the class of PSC metrics with mean-convex boundary.

\appendix

\section{Designing an ad hoc interpolation function}\label{app:GluingFunction} 

In Step 1 of the proof of our main result, Theorem \ref{thm:Desing}, we conveniently employ the following auxiliary function.

\begin{lemma}\label{lem:test.function.for.gluing}
    For every $\varepsilon\in (0,1/2)$, there exists a smooth function $\varphi_{\e}: [0,\varepsilon]\to \R$ such that the following holds:
    \begin{enumerate}[(1)]
        \item $|\varphi_{\e}(t)|\le 2\varepsilon$ for any $t\in [0,\varepsilon]$.
        \item $\varphi_{\e}(t)=0$ in a neighborhood of $0$.
        \item $\varphi_{\e}(t)=\varepsilon$ in a neighborhood of $\varepsilon$.
        \item $(t\varphi_{\e}'(t))'\ge -2\varepsilon$ for any $t\in [0,\varepsilon]$.
        \item $|t\varphi_{\e}'(t)| \le 2\varepsilon^2$ for any $t\in [0,\varepsilon]$.
        \item $\left|t(t\varphi_{\e}'(t))'\right|\le \e$ for any $t\in [0,\varepsilon]$.
    \end{enumerate}
   \end{lemma}

\begin{proof} 

    We construct the function $\varphi_{\e}$ in a few steps. (For the sake of readability, throughout this proof we will write $\varphi$ in lieu of $\varphi_{\e}$, thus leaving the dependence of the parameter $\e\in (0,1/2)$ understood; the same convention is also employed for the functions $f,\hat{f}_1, f_1, \hat{f}_2, f_2$ and $f_0$ coming into play in the sequel of the argument.)
    
    A first key step is to construct a smooth function $f:[0,\varepsilon]\to \R$ such that:
    \begin{enumerate}[(i)]
        \item $f(t)\ge 0$ in $[0,\varepsilon]$.
        \item $f(t)=2\varepsilon$ when $t$ is in a small neighborhood of $0$ and $\varepsilon$.
        \item $\int_0^\varepsilon f(s)ds = 2\varepsilon^2$.
        \item $\int_0^\varepsilon (\log \varepsilon-\log s) f(s)ds =\varepsilon+2\varepsilon^2$.
        \item $f(t)<\frac{\varepsilon}{t}$ for all $t\in [0,\varepsilon]$.
    \end{enumerate}
    
    Towards that goal, let us start by considering the function 
    \[\hat f_1(t)=\begin{cases}
        2\varepsilon^2 e^{\frac{1}{\varepsilon}}, \quad &0\le t\le e^{-\frac{1}{\varepsilon}}\\
        0,& e^{-\frac{1}{\varepsilon}}<t\le \varepsilon.
    \end{cases}\]
    Then we note that $\hat f_1$ satisfies properties (i), (iii) and (v); furthermore
    \[\int_0^\varepsilon (\log \varepsilon - \log s) \hat f_1(s)ds = 2\varepsilon+2\varepsilon^2 + 2\varepsilon^2 \log \varepsilon>\varepsilon+2\varepsilon^2,\]
    provided that $\varepsilon$ is sufficiently small. Therefore, we may slightly modify $\hat f_1$ near the points $0, e^{-\frac{1}{\varepsilon}}, \varepsilon$ and obtain a smooth function $f_1$, which satisfies all properties except for (possibly) property (iv), and in addition
    \[\int_0^\varepsilon (\log \varepsilon-\log s)f_1(s)ds > \varepsilon+2\varepsilon^2,\]
    since this is obviously an open condition.
    Similarly, we then consider the function
    \[\hat f_2(t) = \begin{cases}
        2\varepsilon^2 e^{\frac{1}{3\varepsilon}}, \quad & 0\le t\le e^{-\frac{1}{3\varepsilon}}\\
        0,& e^{-\frac{1}{3\varepsilon}}<t\le \varepsilon.
    \end{cases}\]
    Then $\hat f_2$ satisfies properties (i), (iii) and (v), and 
    \[\int_0^\varepsilon (\log \varepsilon-\log s)\hat f_2(s)ds = \frac{2\varepsilon}{3} + 2\varepsilon^2 +2\varepsilon^2 \log \varepsilon<\varepsilon+2\varepsilon^2,\]
    provided that $\varepsilon$ is sufficiently small. Thus, we may analogously change $\hat f_2$ near $0,e^{-\frac{1}{3\varepsilon}}, \varepsilon$ and obtain a smooth function $f_2$, which satisfies all properties except for property (iv), but instead
    \[\int_0^\varepsilon (\log \varepsilon-\log s)f_2(s)ds <\varepsilon+2\varepsilon^2.\]
    Since all the desired properties are invariant under convex linear combinations of functions, there exists a unique real number $\tau=\tau(\e) \in (0,1)$, such that the corresponding function $f = \tau f_1 + (1-\tau) f_2$ satisfies all properties (i) - (v).

  Fix such $f$. We now define $f_0(t)=-2\varepsilon + f(s)$, a zero average function by (iii), and then in turn
    \[\varphi(t) = \int_0^t \frac 1s \int_0^s f_0(r)drds.\]
    We verify the desired properties of $\varphi$. Clearly $\varphi(0)=0$, and since $f_0(t)=0$ in a small neighborhood of $0$ by (ii) conclusion (2) holds. Assertion (4) is also clearly satisfied, as
    \[t\varphi'(t)=\int_0^t f_0(r)dr\quad \Rightarrow \quad (t\varphi'(t))' = f_0(t)=-2\varepsilon+f(t)\ge -2\varepsilon.\]
    From there we also see that for any $t\in [0,\e]$, by virtue of (i) and (iii)
    \[t\varphi'(t) = \int_0^t f_0(r)dr = -2\varepsilon t + \int_0^t f(r)dr \quad \Rightarrow \quad -2\varepsilon t \le t\varphi'(t)\le -2\varepsilon t + 2\varepsilon^2,\]
    hence conclusion (5) holds. Conclusion (6) holds, as $\frac{t}{\varepsilon}(t\varphi'(t))' = \frac{t}{\varepsilon} (-2\varepsilon + f(t))\leq -2t+1$, where the last inequality exploited the fact that $f$ satisfies property (v).

    Now, the above identity, evaluated at $t=\e$, tells us that $\varepsilon\varphi'(\varepsilon) = \int_0^\varepsilon f_0(s)ds = \int_0^\varepsilon (-2\varepsilon+ f(s))ds = 0$. Thus, $\varphi'(\varepsilon) = 0$. Note also that $(t\varphi'(t))' = f_0(t) = -2\varepsilon + f(t)=0$ in a neighborhood of $\varepsilon$ (again by (ii)), so $\varphi'(t)=0$ in a neighborhood of $\varepsilon$. This implies that $\varphi$ is a constant in a neighborhood of $\varepsilon$.  To determine its actual value, we compute:
    \begin{align*}
        \varphi (\varepsilon) &= \int_0^\varepsilon \frac 1t \int_0^t f_0(s) dsdt
                              = \int_0^\varepsilon f_0(s) \int_s^\varepsilon \frac 1t dt ds \\
                              &= \int_0^\varepsilon (-2\varepsilon + f(s)) (\log \varepsilon - \log s)ds 
                              = \varepsilon
    \end{align*}
    where the last equality relies on (iv).
    Consequently, $\varphi=\varepsilon$ in a small neighborhood of $\varepsilon$, which means that conclusion (3) is verified as well. Finally, we have that $\varphi'(t) = \frac 1t \int_0^t f_0(s)ds$, and since $f_0\ge -2\varepsilon$, we have $\varphi'\ge -2\varepsilon$. Thus, by integrating this inequality over either interval $[0,t], [t,\varepsilon]\subset [0,\varepsilon]$ for all $t\in [0,\varepsilon]$,
    \[\varphi(t) \ge \varphi(0)-2\varepsilon t=-2\varepsilon t,\quad \varphi(t)\le \varphi(\varepsilon)+ 2\varepsilon (\varepsilon-t) \le \varepsilon + 2\varepsilon^2.\]
    Hence conclusion (1) holds, which completes the proof.
    \end{proof}

\section*{Acknowledgements}
The authors wish to express their sincere gratitude to Nick Edelen, Lorenzo Mazzieri, Pengzi Miao, Connor Mooney, Alberto Valli and David Wiygul for several enlightening conversations on themes related to those object of the present manuscript, and for their interest in this work. 
This project has received funding from the European Research Council (ERC) under the European Union’s Horizon 2020 research and innovation programme (grant agreement No. 947923). C.L. was supported by an NSF grant (DMS-2202343) and a Simons Junior Faculty Fellowship.

\bibliography{biblio}

\end{document}